\documentclass{amsart}
\usepackage[utf8]{inputenc}
\usepackage{amsfonts}
\usepackage{amsmath}
\usepackage{amssymb}
\usepackage{amsthm}
\usepackage{colonequals}
\usepackage{xcolor}
\usepackage{bm}
\usepackage{mathtools}
\usepackage{enumitem}
\usepackage{verbatim}
\usepackage{pinlabel}
\usepackage{caption}
\usepackage{subcaption}
\usepackage{tikz}
\usepackage{graphicx}
\usepackage{hyperref}
\usepackage{tikz-cd}
\usepackage{biblatex}
\usepackage{float}
\usepackage{geometry}
\usepackage{soul}
\usepackage{mathtools}

\definecolor{purp2}{RGB}{176,0,255}
\definecolor{purp1}{RGB}{147,0,150}

\newcounter{COMMENTS}
\newcommand{\newComment}[3]{
    \expandafter\newcommand\csname #1\endcsname[1]{ %Defines \Firstname{...comment...}
        \textbf{%
		\color{#3}(\uppercase{#2}\theCOMMENTS)%
        }
        \marginpar{\scriptsize\raggedright\textbf{%
			{\color{#3}(\uppercase{#2}\theCOMMENTS)#1: 
	}} ##1}
	    \stepcounter{COMMENTS}
    }
    \expandafter\newcommand\csname #2\endcsname[1]{{\color{#3} ##1}} %Defines \Initials{...words...}
}

\newComment{Ryan}{rd}{orange}
\newComment{Roberta}{rs}{blue}

\newtheorem{theorem}{Theorem}[section]

\newtheorem{lemma}[theorem]{Lemma}
\newtheorem{proposition}[theorem]{Proposition}

\theoremstyle{definition}

\newtheoremstyle{claim}% name
  {\topsep}% space above
  {\topsep}% space below
  {}% body font
  {}% indent amount
  {\itshape}% theorem head font
  {.}% punctuation after theorem head
  {.5em}% space after theorem head
  {\thmname{#1}\thmnumber{ #2}\thmnote{ (#3)}}% theorem head spec
\theoremstyle{claim}

\DeclareMathOperator{\Diffeo}{Diff}

\DeclareMathOperator{\link}{link}

\DeclareMathOperator{\Star}{star}

\newcommand{\E}{\mathcal{E}}

\newcommand{\fine}{\mathcal{C}^\dagger}
\newcommand{\finearc}{\mathcal{A}^\dagger}

\newcommand{\surf}{S_{g,b}}

\newcommand{\cc}{\mathcal{C}}

\newcommand{\p}[1]{\medskip\noindent\textbf{#1}\textbf{.}}

\newcommand{\pit}[1]{\medskip\noindent\textit{#1}\textit{.}}

\usepackage{tcolorbox} 
\newlist{todolist}{itemize}{2}
\setlist[todolist]{label=$\square$}
\usepackage{pifont}

\title{Homotopy types of fine curve and fine arc complexes}
\author{Ryan Dickmann, Zachary Himes, Alexander Nolte, and Roberta Shapiro}
\date{}

\begin{document}

\begin{abstract}
    The fine curve complex of a surface is a simplicial complex whose vertices are essential simple closed curves and whose $k$-simplices are collections of $k+1$ disjoint curves. We prove that the fine curve complex is homotopy equivalent to the curve complex. We also prove that the fine arc complex is contractible.
\end{abstract}
\maketitle

\section{Introduction}

The \emph{fine curve complex} of the compact surface $S_{g,b}$ with genus $g$ and $b$ boundary components, denoted $\fine(S_{g,b}),$ is the simplicial complex whose vertices are essential nonperipheral simple closed curves in $S_{g,b}$ and $k$-simplices are sets of $k+1$ disjoint curves. The \emph{fine curve graph}, first introduced in the literature by Bowden--Hensel--Webb \cite{BHW}, is the 1-skeleton of the fine curve complex. Bowden--Hensel--Webb studied the fine curve complex and used the full subgraph of the fine curve graph generated by smooth curves to show that $\Diffeo_0(\surf)$ admits unbounded quasimorphisms. This result is in contrast to results on  $\Diffeo_0(M)$ for closed smooth manifolds $M$ with $\mathrm{dim}(M) \neq 2,4$ (see \textcite{burago2008conjugation,tsuboi2008uniform,tsuboi2012uniform}). %added a parathetical because the citations close to the numbers looked weird to me. - Ryan 
A major goal of the study of fine curve graphs is to study classical groups of homeomorphisms with geometric group theoretic tools, as in this result of Bowden--Hensel--Webb.

The fine curve complex differs from the better-studied simplicial complex called the \emph{curve complex}, denoted $\mathcal{C}(\surf)$ (e.g. \textcite{Ivanov,MM}), in that the latter takes as its vertices \emph{isotopy classes} of essential simple closed curves and $k$-simplices are collections of $k+1$ isotopy classes that admit pairwise disjoint representatives. The curve complex of $S_{g, b}$ is a simplicial complex of dimension  $3g+b-3$, whereas the fine curve complex is an infinite dimensional simplicial complex. 
%Homotopy types of curve complexes play a central role in the use of curve complexes to study the mapping class groups of surfaces. Removed because there is a similar sentence later - Ryan

\textcite[Theorem 3.5]{MR830043} and \textcite[Lemma 6.5 and Theorem 6.6]{Ivanov2} showed that the geometric realization of $\cc(\surf)$ is homotopy equivalent to a wedge of $(2g+b-3)$-spheres if $b\neq 0$ or a wedge of $(2g-2)$-spheres if $b=0$.
Homotopy types of curve complexes play a central role in the use of curve complexes to study mapping class groups of surfaces (see, for example, \textcite{HarerHom} and \textcite{MR830043}).

%Throughout, we conflate simplicial complexes with their geometric realizations. 
In Theorem~\ref{thm:cchom}, we relate the homotopy type of the fine curve complex to that of the curve complex. 

\begin{theorem}\label{thm:cchom}
    Let $\surf$ be an orientable surface with $g\geq 1$ or $b\geq 4$. The map $f:\fine(\surf)\to \cc(\surf)$ that sends a curve to its isotopy class is a homotopy equivalence. 
    
    In particular, when $(g, b)\neq(0, b)$ with $b\leq 3$, $\fine(\surf)$ is homotopy equivalent to a wedge of $(2g+b-3)$-spheres if $b\neq 0$ and it is homotopy equivalent to a wedge of $(2g-2)$-spheres if $b=0.$
\end{theorem}

We require $g\geq 1$ or $b\geq 4$ because when $g=0$ and  $b\leq 3$, both $\fine(\surf)$ and $\cc(\surf)$ are empty in these cases, so the theorem is vacuously true.
We note that $(g,b)=(1,0),(1,1),(0,4)$ are odd cases since some versions of (fine) curve complexes allow for curves that intersect at one point (the first two cases) or two points (the last case) to form a simplex.
%be in the same simplices. 
In this paper, we will only consider simplices defined by disjointness. For an analogous result in the setting of 3-manifolds, see \textcite[Proposition 2.3]{MR4898578}.

Alongside the fine curve complex, we have the \emph{fine arc complex.} An \emph{essential arc} is (the image of) a proper embedding $[0,1]\hookrightarrow \surf$ such that the image does not cobound a disk with $\partial \surf.$ The fine arc complex of a surface $\surf$, denoted $\finearc(\surf),$ is the simplicial complex whose $k$-simplices are sets of $k+1$ disjoint essential arcs. The arc complex is the simplicial complex whose $k$-simplices are sets of $k+1$ isotopy classes of properly embedded isotopy classes of essential arcs. \textcite{HarerHom} and \textcite{Hatcher} previously showed that the arc complex is contractible for most orientable surfaces without punctures. In this paper, we prove an analogous result for the fine arc complex.

\begin{theorem}\label{thm:achom}
    Let $\surf$ be an orientable surface with $b\geq 1$ and $(g,b)\neq (0,1), (0,2), (0,3), (0,4), (1,1).$ Then the fine arc complex $\finearc(\surf)$ is contractible.
\end{theorem}

\p{Homotopy types of curve, arc, and finitary curve complexes} \textcite{HarerHom} uses the high connectivity of the curve complex to show that the homology of mapping class groups of surfaces stabilizes. Hatcher's more general theorem in \cite{Hatcher} allows for Harer's results to be extended to more types of vertex sets and larger complexes. 

More recently, the fourth named author showed that the \emph{finitary curve complex}, the simplicial complex whose vertices are the same as that of the fine curve complex but $k$-simplices are collections of $k+1$ curves that pairwise intersect in finitely many points, is contractible (see \textcite{shapiro}). 
%This is the first result on the topology of fine curve complexes.

\p{Outline of paper} In Section~\ref{sec:background}, we introduce some nomenclature surrounding fine curve graphs and prove a key proposition. In Section~\ref{sec:fcc}, we prove Theorem~\ref{thm:cchom}. Along the way, we prove Theorem~\ref{thm:singleisotopyclasscontractible}, which states that the fibers over simplices of the map $f$ from Theorem~\ref{thm:cchom} are contractible. Finally, we prove that the fine arc complex is contractible (Theorem~\ref{thm:achom}) in Section~\ref{sec:fac}.

\p{Acknowledgments} RS thanks Dan Margalit, Daniel Minahan, and Mladen Bestvina for many discussions. ZH thanks Jeremy Miller, Sam Nariman, Peter Patzt, and Jennifer Wilson for helpful conversations.
AN was supported by the National Science Foundation under Grants No. 1842494, 2502952, and 2005551. The authors further thank Jes\'us Hern\'andez Hern\'andez and Yusen Long for comments on a draft.

\section{Key definitions and a critical proposition}\label{sec:background}
The main goal of this section is to prove Proposition~\ref{lemma:keylemma}, which states that given any finite collection of curves or arcs, there is a curve (resp. an arc) arbitrarily close to a given curve (resp. arc) and that intersects every element of the collection only finitely many times at crossing intersections (defined below).
Classical but major results in surface topology play a major role in working with fine curve complexes, and this is the place where such results enter this paper. 
%contains the input from surface topology that we use below.
Proposition~\ref{lemma:keylemma} consolidates the results of Lemma 4.1 and Corollary 4.2 in \textcite{shapiro}. We include a bare-bones proof with pointers to corresponding results and lemmas in \cite{shapiro} as well for completeness.

Throughout the paper, we adopt the convention that all curves and arcs are simple, i.e. have no self-intersections. Whenever we refer to an arc, we will implicitly mean an essential arc. Similarly, whenever we refer to a curve in a surface, we will implicitly assume that the curve is essential (meaning that it is not nullhomotopic) and nonperipheral (meaning that it is not homotopic to a boundary component of the surface).

We first introduce some nomenclature for points of $u\cap v$ that are not accumulation points for the set $u\cap v$.

   Two curves $u$ and $v$ are \emph{touching} at a point $c\in u\cap v$ if, in an arbitrarily small neighborhood $N$ of $c,$ $u$ and $v$ can be isotoped to be disjoint. Otherwise, $u$ and $v$ are \emph{crossing} at $c.$  
Examples of such intersections can be found on the right and left, respectively, of Figure~\ref{fig:crossingvstouching}. 

\begin{figure}[h]
\begin{center}
\includegraphics[width=3in]{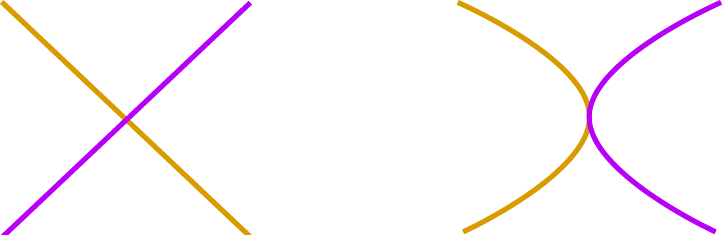}
\caption{Left: a crossing intersection. Right: a touching intersection.}\label{fig:crossingvstouching}
\end{center}
\end{figure}

\begin{proposition}\label{lemma:keylemma}
    Let $S=\surf$ be an orientable surface and let $\gamma_1,\ldots,\gamma_n$ be a collection of curves or arcs in $S.$ Let $y$ be a curve (resp. an arc) in $S.$ Then there exists a curve (resp. an arc) $\gamma$ in $S$ such that:
    \begin{enumerate}
        \item $|\gamma\cap \gamma_i| < \infty$ for all $i,$
        \item all intersections $\gamma\cap \gamma_i$ for all $i$ are crossing, and
        \item $\gamma$ is contained in some arbitrary pre-selected open neighborhood of $y$.
    \end{enumerate}
\end{proposition}

The idea of the proof is that we will  construct $\gamma$ by removing a certain compact and totally disconnected set $\mathcal{E}\subset \surf$ of ``problematic" intersections of $\gamma_i\cap \gamma_j$ from $\surf$.
%and showing that $\surf\setminus \mathcal{E}$ is path-connected.
We will show that $\mathcal{E}$ is totally disconnected by using a result from dimension theory, which we now provide some background information on (for more details on dimension theory, see, for example, \textcite{HurewiczDimensionTheory} or \textcite{Engelking1978}).

%We begin with some background information on zero-dimensional spaces.
A space $X$ is \emph{zero-dimensional} if it has a basis consisting of sets which are both closed and open in $X$. If $X$ is zero-dimensional, then it is totally disconnected and the converse holds if $X$ is locally compact and Hausdorff (see, for example, \textcite[Proposition 3.1.7]{topgroupsTkachekno}).
An example of a zero-dimensional space is any non-empty subspace of $\mathbb{R}$ which does not contain an interval (for more details, see \textcite[1.2.5 Examples]{Engelking1978}).
%the two notions are equivalent. 
%being zero-dimensional is equivalent to being totally disconnected  
We will need the following theorem about unions of zero-dimensional spaces in the proof of Lemma~\ref{lem: E is tot discon}.
%(Theorem II 2 The sum theorem for 0-dimensional spaces, Dimension theory)
\begin{lemma}[{e.g. \textcite[Theorem II 2]{HurewiczDimensionTheory}}]\label{thm:sum of zero-dims}
    A space which is the countable union of closed zero-dimensional subspaces is itself zero-dimensional.
\end{lemma}
%of clopen sets (recall that a subset $U\subseteq X$ is clopen if it is both closed and open).
%such that every element in this basis is a closed set (equivalently, $X$ has a b
%consisting of closed-open sets (

%For this proof, we will need 
%to prove several topological properties about $\mathcal{E}$
%to show that $\mathcal{E}$ is a closed and totally disconnected subspace of $S_{g,b}$.
%To show this, we need the following result (cite Kechris).
%We will show that this set of problematic intersections is a closed subset of a Cantor set and then working in the remaining surface.

We now define the set of ``problematic" intersections between a finite set of curves or arcs.
%We will now define the set of ``problematic" intersections between a finite set of curves or arcs.
Given a collection of distinct curves or arcs $\gamma_1,\ldots, \gamma_n$ in a surface $\surf,$ we define $\mathcal{E}(\gamma_1,\ldots,\gamma_k)\subset \surf$ in the following way. We will abuse notation by treating the $\gamma_i$ both as injective maps from $S^1 = I/(\{0\}\sim\{1\})$ or $I$ into $\Sigma$ and as the images of such maps; the use should be clear from context.
    
Let $E=\{p\in \surf \mid p\in \gamma_i\cap \gamma_j \text{ for some }i,j\};$ these are all the points of intersection between the arcs and curves. Intersections where two curves or arcs coincide for an open interval pose no problem to us, so we wish to exclude them. To this end, let $P_i=\partial(\gamma_i^{-1}(E))\subset I;$ we note that $P_i$ is empty if $\gamma_i$ is disjoint from all other $\gamma_j$. To get these points back on the surface, let $\mathcal{E}(\gamma_1,\ldots,\gamma_k)=\bigcup_i \gamma_i(P_i).$ We will often simplify notation and write $\mathcal{E}$ for $\mathcal{E}(\gamma_1,\ldots,\gamma_k)$.

%Lemma: $\mathcal{E}$ is compact, separable, and metrizable.

%To show that $\mathcal{E}$ is totally disconnected, we will use certain results from dimension theory (for more details, see , for example). A space $X$ is zero-dimensional if. If $X$ is Hausdorff, this condition is equivalent to. An example of a zero-dimensional space is.

%To show that  $\mathcal{E}$ is totally disconnected, it suffices to show that $\mathcal{E}$ is zero-dimensional. In order to do this, we need the following result.

%Theorem:
\begin{lemma}\label{lem: E is tot discon}
  The space $\E$ is compact and totally disconnected.
  %As a result, $\E$ is a closed subset of a Cantor set.
\end{lemma}
\begin{proof}
    The space $\E$ is compact because it is a finite union of compact spaces.

    To show that $\E$ is totally disconnected, it suffices to show that it is zero-dimensional. First, we show that $\gamma_i(P_i)$ is zero-dimensional. Since $P_i$ is the boundary of a subspace of $I$ (and therefore $\mathbb{R}$), it cannot contain an interval and so it
    %follows from \cite[Example 1.2.5]{Engelking1978} that $P_i$ 
    is zero-dimensional. Moreover, since $\gamma_i$ is an embedding, $P_i$ is homeomorphic to $\gamma_i(P_i).$ Homeomorphisms preserve dimension, so $\gamma_i(P_i)$ is zero-dimensional. Since $\E=\bigcup_{i}\gamma_i(P_i)$ is a finite union of closed zero-dimensional spaces, by Lemma~\ref{thm:sum of zero-dims},
%\cite[Theorem 1.3.1]{Engelking1978}, 
$\E$ itself is zero-dimensional. We conclude that $\E$ is totally disconnected.
\end{proof}
%Lemma: $\mathcal{E}$ is compact and zero-dimensional.

%Lemma: $\mathcal{E}$ is a closed subset of a Cantor set.

%\begin{lemma}
%    The space $\surf\setminus\mathcal{E}$ is path-connected.
%\end{lemma}
%\begin{proof}
%    As a consequence of , $\surf\setminus\mathcal{E}$ is connected. Since $\surf$ is locally path-connected and $\surf\setminus\mathcal{E}$ is a connected open subspace of $\surf$, the space $\surf\setminus\mathcal{E}$ is path-connected.
%\end{proof}
%Do (resp. arc).

Part of the proof of Proposition~\ref{lemma:keylemma} involves working with a  regular neighborhood of a given curve or arc $y$ in $S_{g,b}$. If $y$ is a curve, a regular neighborhood is homeomorphic to an annulus. If $y$ is an arc, such a neighborhood is a strip in $\surf$. 
More formally, a \emph{strip} in $\surf$  is an embedded copy of $I\times I$ where $\{0\}\times I,\{1\}\times I \subseteq \partial\surf.$ We will assume that $I\times \{0\}$ does not bound a disk with $\partial \surf$.

In the proof of Proposition~\ref{lemma:keylemma}, we will also need the following lemma. 
\begin{lemma}\label{lemma:finiteintannulus}
    Let $S=\surf$ and $\gamma_1,\ldots,\gamma_n$ be a collection of curves or arcs in $\surf.$ Let $A$ be a closed annulus (resp. a strip) in $\surf$ such that, for all $i\neq j,$ the intersection $\mathrm{Int}(A)\cap \gamma_i\cap \gamma_j$ is a union of connected components of $\gamma_i \cap \mathrm{Int}(A)$. Then, there exists a curve (resp. an arc) $\gamma \subset A$ 
    %(if $A$ is an annulus) or an arc $\alpha\subset A$ (if $A$ is a strip) 
    such that $|\gamma_i\cap \gamma|<\infty$ or $|\gamma_i\cap \alpha|<\infty$  for all $i$ and all such intersections are crossing.
\end{lemma}

\begin{proof}
    This lemma is proven identically to Lemma 3.3 of Shapiro \cite{shapiro} despite having more curves in the collection.
    The key input is the tameness of arcs in $\mathbb{R}^2$ (see, for example, \textcite[Ch. 10]{moise1977geometric}).
\end{proof}

We will also need the following result to prove Proposition~\ref{lemma:keylemma}.
%To show that $\surf\setminus\mathcal{E}$ is path-connected, we need the following result.
\begin{lemma}[{\textcite[Chapter V, Theorem 14.3]{newman1951elements}}]\label{complementoftotdisconnected}
    Let $X=\mathbb{R}^{2}$ or the closed disk $\bar{D}^{2}$ and let $V$ be a closed subspace of $X$. Suppose that two points $x, y$ in $X$ are in different connected components of $X\setminus V$. Then there is a connected component $C$ of $V$ such that $x$ and $y$ are in different components of $X\setminus C$.
    %Suppose that two $x$ and $y$ in $\mathbb{R}^{2}$ are separated by a closed subspace $V$ of $\mathbb{R}^{2}$ (by this, we mean that $x$ and $y$ are in different connected components of $\mathbb{R}^{2}\setminus V$). Then $x$ and $y$ are separated by a connected component of $V$.
\end{lemma}

\begin{proof}[Proof of Proposition~\ref{lemma:keylemma}]
    Let $N$ be a regular neighborhood of $y$ that is closed in $\surf$. We will construct $\gamma$ in such a way that it is contained in $N$.
    %which we want $\gamma$ to be contained. 
    The space $N$ is homeomorphic to an annulus or a strip, depending on whether $y$ is a curve or an arc. Let $\E=\E(\gamma_1,\ldots,\gamma_n)$ and let $N^*=N\setminus \E.$ Since $\E$ is totally disconnected by Lemma~\ref{lem: E is tot discon}, each connected component of $\E$ is a point. Lemma~\ref{complementoftotdisconnected} then implies that $N^{*}$ is connected. Since $N$ is locally path-connected and $N^{*}$ is a connected open subspace of $N$, the space $N^{*}$ is path-connected.
    %Then $N^*$ is path-connected.
    
%    Since $E$ is closed by Lemma~\ref{lem: E is tot discon}, $N^{*}$ is an open subspace of $N$. In addition, since $N$ is locally-path connected and $N^{*}$ is open, $N^{*}$ is locally path-connected.
    
    %Since $\E$ is a closed and totally disconnected subset of a surface, it follows from \cite[Thm. 14.2]{newman1951elements} that $N^*$ is path-connected: this is the key point.
 %   Since $\E$ is compact and totally disconnected by Lemma~\ref{lem: E is tot discon}, $N\cap \E$ is a compact and totally disconnected subset of $N$. Therefore, by embedding $N$ into $S^{2}$, we may apply Alexander duality to $N\cap \E\subset N$
    %by an application of Alexander duality (applied to $N\cap \E\subset N$ and with $N$ embedded in $S^{n}$), 
  %  to conclude that $N^*$ is connected. Since a locally path-connected is path-connected if and only if it is connected, $N^*$ is path-connected. We now will show that $N^*$ is path-connected.
    %In the following, we conflate $N$ and $N^*$ by considering the natural inclusion $N^*\hookrightarrow N$ (and its preimage).

   Suppose $N$ is a strip. Then there exists an arc $a$ in $N^*$ connecting the components of $\partial \surf$ in $N$. We may take a regular neighborhood of $a$ to obtain a strip $A$ in $N^*$. We can now apply Lemma~\ref{lemma:finiteintannulus} to find an arc $\gamma$ in $N$ that intersects all $\gamma_i$ finitely many times.

    Suppose $N$ is an annulus. There exists a point in each component of $\partial N$ that is not in $\E.$ Consider these two points as points in $N^*$ and connect them with an arc $\alpha$ in $N^*$. Let $x\in \alpha,$ and then cut $N^*$ along $\alpha$ to form a strip (that has two copies of $\alpha$). Since the strip is also a connected surface, there is an arc $c$ connecting the two copies of $x.$ Reglue the strip into $N^*$; the endpoints of $c$ glue up into a curve $\gamma'$. 
    %Call the inclusion of this curve into $\surf$ $\gamma.$ 
    Let $A$ be a regular neighborhood of $\gamma'$ inside $N^*$. Since $A$ is an annulus, we can apply Lemma~\ref{lemma:finiteintannulus} to find a curve $\gamma$ in $N$ that intersects each $\gamma_i$ finitely many times.
\end{proof}

\section{Homotopy type of the fine curve complex}\label{sec:fcc}

The goal of this section is to prove Theorem~\ref{thm:cchom}. To this end, we first prove that
the fiber over a simplex $\sigma\in \cc(\surf)$
under the map $f:\fine(\surf)\to \cc(\surf)$
%the subcomplex of $\fine(\surf)$ induced by a simplex in the curve complex 
is contractible (Section~\ref{subsec:SIC}) and then complete the proof of Theorem~\ref{thm:cchom} (Section~\ref{subsec:proofofcchom}). We begin in Section~\ref{sec:simpcomplex} by recalling some background information and fixing some conventions involving simplicial complexes. For more on simplicial complexes, see, for example, \textcite[Sections 1.1--1.3]{MunkresAlgTop}.

\subsection{Simplicial complexes}\label{sec:simpcomplex}
Given a simplicial complex $X$ and a simplex $\sigma\in X$, we  denote the dimension of $\sigma$ by $\dim(\sigma)$. We denote the link and (closed) star of $\sigma$ by $\link(\sigma)$ and $\Star(\sigma)$ respectively. We denote the $k$-skeleton of $X$ by $X^{(k)}$.

We denote the geometric realization of $X$ by $|X|$. When we say that a simplicial complex $X$ has a topological property, we will mean that $|X|$ has that property.

A point $x\in |X|$ is represented in barycentric coordinates by a formal sum $\sum_{v\in X^{(0)}} t_{v} v$,with $t_{v}\in [0,1]$, all but finitely of the $t_{v}$s equal to zero (with the set of such $v$ with $t_{v}$ non-zero forming a simplex in $X$), and $\sum_{v\in X^{(0)}}t_{v}=1$. 

 Let $X$ and $Y$ be simplicial complexes. A map of spaces $g\colon |X|\to |Y|$ is \emph{simplicial} if it is induced from a simplicial map of simplicial complexes $X\to Y$.
In this case, we will abuse notation and also use $g$ to denote the underlying simplicial map $X\to Y$  inducing $g\colon |X|\to |Y|$. 

For $k\geq 0$, a \emph{combinatorial k-manifold} $M$ is a nonempty simplicial complex that satisfies the following inductive property. If $\sigma\in M$, then $\dim(\sigma)\leq k$. In addition, if $k-\dim(\sigma)-1\geq 0$, then $\link(\sigma)$ is a combinatorial $(k-\dim(\sigma)-1)$-manifold homeomorphic to either a $(k-\dim(\sigma)-1)$-sphere or a closed $(k-\dim(\sigma)-1)$-ball. A combinatorial $k$-manifold homeomorphic to a $k$-sphere will be called a \emph{combinatorial k-sphere}. We will often denote a combinatorial $k-$sphere by $\mathbb{S}^{k}$.
    %A \emph{combinatorial triangulation} of $S^{n}$ is a simplicial complex whose geometric realization is piecewise-linear homeomorphic to $S^{n}$.
%Let $\sigma\in X$.  We denote the (closed) star and link of $\sigma$ in $X$ by $\Star(\sigma)$ and $\link(\sigma)$ respectively. 
%\begin{definition}
%    For $k\geq 0$, a \emph{combinatorial k-manifold} $M$ is a non-empty simplicial complex that satisfies the following inductive property. If $\sigma\in M$, then $\dim(\sigma)\leq k$. In addition, if $k-\dim(\sigma)-1\geq 0$, then $\link(\sigma)$ is a combinatorial $(k-\dim(\sigma)-1)$-manifold homeomorphic to either a $(k-\dim(\sigma)-1)$-sphere or a closed $(k-\dim(\sigma)-1)$-ball. A combinatorial $k$-manifold homeomorphic to a $k$-sphere will be called a \emph{combinatorial k-sphere}.
    %A \emph{combinatorial triangulation} of $S^{n}$ is a simplicial complex whose geometric realization is piecewise-linear homeomorphic to $S^{n}$.
%\end{definition}
%By abuse of notation, if $S$ is a combinatorial $k$-sphere, we will abuse notation and also denote its geometric realization by $S$. It will be clear from context when we are working with the underlying simplicial complex or its geometric realization.

By the simplicial approximation theorem (see, for example, \textcite[Theorem 16.1]{MunkresAlgTop}), every element of $\pi_{k}(|X|)$ is represented by a simplicial map
$|\mathbb{S}^{k}|\to |X|$.
%of simplicial complexes from a combinatorial $k$-sphere $\mathbb{S}^{k}$ to $X$. 
For this reason, for the rest of the paper, whenever we work with a map $g\colon S^{k}\to |X|$, we will implicitly treat $S^{k}$ as a combinatorial $k$-sphere  $\mathbb{S}^{k}$ and assume that $g\colon |\mathbb{S}^{k}|\to |X|$ is a simplicial map.

In the proof of Theorem~\ref{thm:cchom}, we will need the following proposition, which gives a condition for when a simplicial map of simplicial complexes is a homotopy equivalence.
%(it is a special case of Quillen's Theorem A). {\color{red}CITE!!!}
\begin{proposition}[e.g. {\textcite[Corollary 2.7]{HV}}]\label{prop: contractible fib}
    Let $g\colon X\to Y$ be a simplicial map of simplicial complexes. If the subcomplex $g^{-1}(\sigma)$ is contractible for all simplices $\sigma\in Y$, then $g$ is a homotopy equivalence.
\end{proposition}
\subsection{Preimages of simplices are contractible}\label{subsec:SIC} The goal of this section is to prove Theorem~\ref{thm:singleisotopyclasscontractible}, which states that the preimage of a simplex in $\cc(\surf)$ is contractible.

   Given a simplex $A$ in $\cc(\surf)$, let $\mathcal{C}_A^\dagger(S_{g,b})$ denote the subcomplex $f^{-1}(A)$ in $ \fine(\surf)$. A simplex  $\{v_{0},\ldots, v_{q}\}\in \fine(\surf)$ is in $\mathcal{C}_A^\dagger(S_{g,b})$ if $f(v_{i})$ is a vertex in $A$ for all $i$.
%Let $A$ be a collection of isotopy classes of curves that admit pairwise disjoint representatives, so that $A$ induces a simplex in $\cc(\surf).$ Denote the subcomplex of $\fine(\surf)$ spanned by all representatives of the isotopy classes in $A$ by $\fine_A(\surf).$ In particular, $\fine_A(\surf)$ is the preimage of $A$ under the natural collapsing map: $\mathcal{C}_A^\dagger(S_{g,b})= f^{-1}(A)$. 

\begin{theorem}\label{thm:singleisotopyclasscontractible}
    Let $\surf$ be an orientable surface with $g\geq 1$ or $b\geq 4$ and let $A$ be a simplex in $\cc(\surf)$. 
    %Let $A$ be a collection of isotopy classes of curves admitting disjoint representatives.
    Then $\fine_A(\surf)$ is contractible.
\end{theorem}

% Let $[\alpha]$ be an isotopy class of simple closed curves in $\surf$. Define the single-isotopy-class fine curve complex associated to $[\alpha]$, denoted $\fine_\alpha(\surf),$ to be the subcomplex of $\fine(\surf)$ induced by all elements of $[\alpha].$ In this section, we prove that $\fine_\alpha(\surf)$ is contractible.

% \begin{theorem}\label{thm:singleisotopyclasscontractible}
%     Let $\surf$ be an orientable surface with $g\geq 1$ or $b\geq 4$. Let $\alpha$ be a curve in $\surf.$ Then $\fine_\alpha(\surf)$ is contractible. {\color{red}actually, our proof says that the fiber over any simplex downstairs is contractible!}
% \end{theorem}

Our main tools are Proposition~\ref{lemma:keylemma} and Whitehead's theorem. To this end, we first show that any map $S^{k}\to |\fine_A(\surf)|$ is homotopic to a map whose image is  contained in simplices spanned by curves with desirable intersection properties (Lemma~\ref{lemma:nicesphere}). Then, we show that any map $S^{k}\to |\fine_A(\surf)|$ can be further homotoped to a map whose image is contained in the star of a single vertex (Lemma~\ref{lemma:spheretostar}). 
\begin{lemma}\label{lemma:nicesphererepcurves}
    Let $\surf$ be an orientable surface and $\Gamma=\{v_1,\ldots,v_n\}$ be a collection of curves in $\surf.$ Then there exists a collection of curves $v_1',\ldots,v_n'$ such that, for all $i$ and $j,$
    \begin{enumerate}
        \item $v_i\cap v_i'=\emptyset$ or $v_i=v'_i$,
        \item if $v_i\cap v_j = \emptyset,$ then: 
        \begin{enumerate}
            \item $v_i'\cap v_j' = \emptyset$ and
            \item $v_i'\cap v_j = \emptyset$, and
        \end{enumerate}
        %and $v_i'\cap v_j=\emptyset$ if $i<j,$ and
        \item $|v'_i \cap v'_j|<\infty$ and all intersections are crossing.
    \end{enumerate}
\end{lemma}

\begin{proof}
    Let $v'_1=v_1$ and let $\Gamma_1=\Gamma=\Gamma\cup \{v_1'\}.$ Consider $v_2.$ 
    We may pick an open regular neighborhood $T$ of $v_{2}$ such that if $v_j$ does not intersect $v_{2}$, then $v_j$ also does not intersect $T$. Pick a curve $y$ in $T\setminus v_{2}$. 
    By Proposition~\ref{lemma:keylemma}, 
    %Then, by Proposition~\ref{lemma:keylemma} and by choosing an annulus disjoint from and arbitrarily close to $v_2$ for the construction, 
    there exists a curve $v'_2$ contained in the neighborhood $T\setminus v_{2}$ of $y$ such that:
    \begin{enumerate}
        \item $v'_2\cap v_2 = \emptyset,$
        \item if $v_j \cap v_2 = \emptyset,$ then
         $v_j\cap v'_2=\emptyset$ and if $v'_j$ has already been defined, $v'_j\cap  v'_2=\emptyset$, and
        %\begin{enumerate}
        %    \item if $i<2,$ $v_i'\cap v_2' = \emptyset$ and
        %    \item if $i>2,$ $v_2'\cap v_i = \emptyset$,
        %\end{enumerate}
        \item for $i < 2$, we have $|v_2' \cap v_i'| < \infty$ and all intersections of $v_2'$ and $v_i'$ are crossing.
    \end{enumerate}
    Condition $(1)$ follows from the fact that $v'_2$ is in $T\setminus v_{2}$, condition $(2)$ follows from our choice of $T$, and condition $(3)$ follows from Proposition~\ref{lemma:keylemma}.
    %disjoint from $v_2$ such that: 1) $|v'_2\cap v'_i|<\infty$ for $i<2$, 2) $|v'_2\cap v_i|<\infty$ for $i>2,$ 3) $v'_2\cap v'_i=\emptyset$ for $i<2$ if $v_2\cap v_i=\emptyset,$ and 4) $v'_2\cap v_i=\emptyset$ for $i>2$ if $v_2\cap v_i=\emptyset.$ 

    We now consider the set of curves $\Gamma_2 = \{v_1,v_2,\ldots,v_n\}\cup \{ v_1',v_2'\}=\Gamma_1\cup \{v_2'\}$ and perform the above algorithm once again, this time constructing $v'_3$ using $v_3$ and letting $\Gamma_3=\Gamma_2\cup \{v_3'\}.$

    We continue similarly, constructing $v'_i$ using $v_i\in\Gamma_{i-1}$ and letting $\Gamma_i=\Gamma_{i-1}\cup\{v'_i\}.$

    In this manner, we have constructed a collection of curves $\Gamma'=\{v_1',\ldots,v_n'\}$ that satisfy the conditions in the lemma statement.
    \end{proof}

We have the following consequence of Lemma~\ref{lemma:nicesphererepcurves}.

\begin{lemma}\label{lemma:nicesphere}
    Any map $\phi\colon S^k\to |\fine_A(\surf)|$ is homotopic to a simplicial map $\phi'\colon S^{k}\to |\fine_A(\surf)|$ whose image lies in the full subcomplex spanned  by a set of curves $c_1,\ldots,c_n$ such that $|c_i\cap c_j|<\infty$ for all $i,j$ and all intersections between pairs of curves are crossing.
    %In addition, $\phi'$ is simplicial with respect to some combinatorial triangulation of $S^{k}$.
    %The image of any continuous map $\phi:S^k\to \fine_A(\surf)$ is homotopic to a sphere whose image lies in simplices spanned by a set of curves $c_1,\ldots,c_n$ such that $|c_i\cap c_j|<\infty$ for all $i,j$ and all intersections between pairs of curves are crossing.
\end{lemma}

\begin{proof}
    The proof idea is to flow $\phi(S^k)$ into simplices spanned by the curves $c_i$.
    %%The simplicial approximation theorem implies that any continuous map $\phi: S^{k}\to |\fine_A(\surf)|$ is homotopic to a map $\phi'\colon |S^k|\to |\fine_A(\surf)|$ that is simplicial with respect to some combinatorial triangulation $|S^k|$ of $S^k$. As a result, we may assume that $S^{k}$ comes equipped with some combinatorial triangulation and that $\phi: S^k\to |\fine_A(\surf)|$ is simplicial with respect to this combinatorial triangulation of $S^k$.
    Recall that, without loss of generality, we can assume $\phi\colon S^{k}\to |\fine_A(\surf)|$ is induced by a simplicial map $\mathbb{S}^{k}\to\fine_A(\surf) $.
    Since $S^k$ is compact and $\phi$ is continuous, the image of $\phi$ is contained in the interior of simplices spanned by a finite number of vertices $b_1,\ldots, b_n$. 
     By Lemma~\ref{lemma:nicesphererepcurves}, there are curves $c_1\ldots,c_n$ such that
     for all $i$ and $j$, either $c_{i}\cap b_{i}=\emptyset$ or $c_{i}=b_{i}$; if $b_{i}\cap b_{j}=\emptyset$ then $c_{i}\cap c_{j}=\emptyset$ and $c_i\cap b_j=\emptyset$; and $|c_{i}\cap c_{j}|<\infty$ and all intersections are crossing. For each vertex $v\in (\mathbb{S}^{k})^{(0)}$, let $b_{v}$ denote $\phi(v)$ and let $c_{v}$ be the corresponding curve in $c_{1},\ldots, c_{n}$ (i.e. if $\phi(v)=b_{i}$, then $c_{v}\colonequals c_{i}$). The map $\phi'\colon \mathbb{S}^{k}\to \fine_A(\surf)$ that sends $v$ to $c_{v}$ is a simplicial map since $\phi\colon \mathbb{S}^k\to \fine_A(\surf)$ is a simplicial map. 
     
     We now show that $\phi\colon |\mathbb{S}^{k}|\to |\fine_A(\surf)|$ is homotopic to the simplicial map $\phi'\colon |\mathbb{S}^{k}|\to |\fine_A(\surf)|$. Consider the map 
    \begin{align*}
        H \colon |\mathbb{S}^{k}| \times I &\to |\fine_A(\surf)|\\
         %\big(\sum_{v\in (\mathbb{S}^{k})^{0}} t_{v} v, s\big)&\mapsto \sum_{v\in (\mathbb{S}^{k})^{0}} (1-s)t_{v} b_{v}+ st_{v} c_{v}\\
         \left(\sum_{v\in (\mathbb{S}^{k})^{0}} t_{v} v, s\right)&\mapsto (1-s)\sum_{v\in (\mathbb{S}^{k})^{0}} t_{v} b_{v}+ s\sum_{v\in (\mathbb{S}^{k})^{0}} t_{v} c_{v}.
    \end{align*}
    The formal sum  %$\sum_{v\in (\mathbb{S}^{k})^{0}} (1-s)t_{v} b_{v} + st_{v} c_{v}$ 
    $(1-s)\sum_{v\in (\mathbb{S}^{k})^{0}} t_{v} b_{v}+ s\sum_{v\in (\mathbb{S}^{k})^{0}} t_{v} c_{v}$ is a point in $|\fine_A(\surf)|$ and the map $H$ is continuous because if $\{b_{v_{0}},\ldots, b_{v_{p}}\}$ is a simplex in $\fine_A(\surf)$, then $\{b_{v_{0}}, c_{v_{0}},\ldots, b_{v_{p}}, c_{v_{p}}\}$ is also a simplex.
    %in  does not intersect $b_{w}$, then $c_{v}$ does not intersect $c_{w}$.
    Since $H(-, 0)=\phi$ and the map $H(-, 1)=\phi'$  has its image contained in the full subcomplex spanned by $c_{1},\ldots, c_{n}$, the result follows.  
\end{proof}

%We now recall the definition of a bigon in a surface because we will use them to show that any map $\phi\colon S^{k}\to|\fine_A(\surf)|$ is homotopic to a map whose image is contained in the star of a vertex. 
Now that we have that a map $|\mathbb{S}^k|\to |\fine_A(\surf)|$ can be homotoped to one whose image is spanned by curves that pairwise intersect finitely many times, we are ready to remove the remaining intersections.
A \emph{bigon} in a surface $S$ is a disk bounded by exactly two arcs, with each arc belonging to a different curve. For more details on bigons, see Section 1.2.4 of \textcite{primer}.
%begin{definition}
%    A  \emph{bigon} in a surface $S$ is a disk bounded by exactly two arcs, with each arc belonging to a different curve.
%\end{definition}
%For more details on bigons, see Section 1.2.4 of \textcite{primer}.

If two curves $c_{1}$ and $c_{2}$ in a surface $S$ have the property that they intersect a finite number of times with only crossing intersections and their isotopy classes $[c_{1}]$ and $ [c_{2}]$ form an edge in $\cc(S)$, then $c_{1}$ is homotopic to a curve $c_{1}'$ disjoint from $c_{2}$. Furthermore, $c_{1}'$ can be constructed from $c_{1}$ by removing all bigons between $c_{1}$ and $c_{2}$ (for more details, see the first proof of Proposition 1.7 in \textcite{primer}). This will be the main idea behind the proof of the following lemma.

\begin{lemma}\label{lemma:spheretostar}
    Let $\phi\colon S^k\to |\fine_A(\surf)|$ be a map. Then $\phi$ is homotopic to a simplicial map $\phi'\colon S^k\to |\fine_A(\surf)|$ whose image is in the star of a vertex.
    %Then, $\phi(S^k)$ can be homotoped to be contained in the star of one vertex.
\end{lemma}

\begin{proof}
    By Lemma~\ref{lemma:nicesphere}, we can assume that all vertices that span $\phi(S^k)$ have only crossing intersections (and a finite number of them); call these curves $c_1,\ldots,c_n$. We will show that $\phi$ is homotopic to a simplicial map whose image  is in $\Star(c_{n})$.
    %We will homotope $\phi(S^k)$ to $\Star(c_n).$

    We describe an algorithm that decreases total intersection number with $c_k$ by removing bigons one at a time. Take an innermost bigon of all the curves $c_1,\ldots,c_{n-1}$ with $c_n.$ There may be more than one, and some bigons may not be comparable, but that does not matter. Without loss of generality, this bigon is formed by $c_1$ and $c_n.$ Make a copy of $c_1$ to the side of $c_1$ that is inside the bigon such that the copy has the same intersection patterns with all the other curves as $c_1;$ call this curve $c_1'.$ This is possible since all intersections between curves are crossing. We note that since the bigon we are considering is innermost, any curve $c_i$ disjoint from $c_1$ is also disjoint from the bigon. (A curve that intersects the bigon but not $c_1$ would form a bigon with $c_n$ contained in the original bigon, contradicting the original bigon being innermost.) Now, surger $c_1'$ along $c_n$ so that it no longer forms the bigon in question. A visual of this can be seen in Figure~\ref{fig:contractibilityalgorithm}. By using a homotopy similar to the one in the proof of Lemma~\ref{lemma:nicesphere}, we have that $\phi\colon S^k\to |\fine_A(\surf)|$ is homotopic to a simplicial map $\phi'\colon S^k\to |\fine_A(\surf)|$ whose image is contained in the full subcomplex spanned by $c_{1}', c_{2},\ldots, c_{n}$.

    \begin{figure}[h]
    \begin{center}
    \begin{tikzpicture}
    \node[anchor = south west, inner sep = 0] at (0,0) {\includegraphics[width=0.3\textwidth]{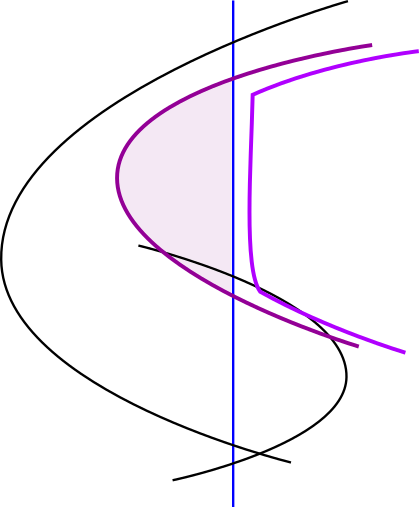}};
    %\draw[help lines] (0,0) grid (5,5);
    \node at (2.35,5.4){{\color{blue}$c_n$}};
    \node at (3.5,4.5){{\color{purp2}$c_1'$}};
    \node at (1.1,3.5){{\color{purp1}$c_1$}};
    \node at (3.75,0.9){$c_2$};

    \end{tikzpicture}
    \caption{Our algorithm involves finding an innermost bigon with $c_n$, such as the pictured shaded one formed by $c_1$, and homotoping across it. An alternative innermost bigon is the one created by $c_2.$}\label{fig:contractibilityalgorithm}
    \end{center}
    \end{figure}
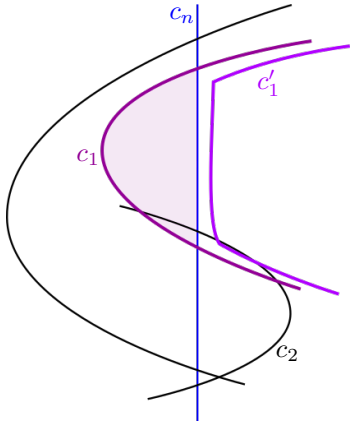
    
    %We notice that if $c_1$ is a vertex of a simplex $S$ intersecting the image of $\phi(S^k)$, then $S \cup \{c_1'\}$ also spans a simplex in $\mathcal{C}^\dagger_{A}(S_{g,b})$. By considering any point in a simplex containing $c_1$ as actually being contained in the face of a higher dimensional simplex, we can linearly transform any such point to the face with $c_1'$ instead of $c_1$ (as in the proof of Lemma~\ref{lemma:nicesphere}).

    We now notice the following: total intersection number with $c_n$ has decreased since $|c_1'\cap c_n|=|c_1\cap c_n|-2$, the total intersection number with $c_n$ is finite, and the minimal intersection number between all the vertices and $c_n$ is 0 by the definition of $A$ (which can necessarily be achieved by removing bigons). 

    By repeatedly applying the above bigon removal, we will have removed all intersections with $c_n.$ It follows that $c_{1}',\ldots,c_{n-1}' \in \Star(c_n)$ (where the $c_i'$ may have had to go through multiple rounds of surgery and replacement).
    
    %Therefore, since $c_{1}'$ intersects with $c_{2},\ldots, c_{n-1}$ finitely many times and all intersections are crossing, after repeated application of this algorithm, we will have removed all intersections with $c_n,$ so $c_{1}',\ldots,c_{n-1}' \in \Star(c_n)$ (where the $c_i'$ may have had to go through multiple rounds of surgery and replacement). 
    Therefore, $\phi\colon S^k\to |\fine_A(\surf)|$ is homotopic to a simplicial map whose image  is in $\Star(c_{n})$.
    %Therefore, $\phi(S^k)$ can be homotoped to the star of a vertex.
\end{proof}

We are now ready to prove Theorem~\ref{thm:singleisotopyclasscontractible}.

\begin{proof}[Proof of Theorem~\ref{thm:singleisotopyclasscontractible}]
    % We will prove the theorem using Whitehead's theorem. Let $\phi\colon S^k\to |\fine_A(\surf)|$ be a continuous map. By Lemma~\ref{lemma:spheretostar}, $\phi$ is homotopic to a simplicial map $\phi'\colon S^k\to |\fine_A(\surf)|$ whose image is in the star of a vertex. Since the star of a simplex in a simplicial complex is contractible,  $\phi'$ is nullhomotopic. 

    % {\color{blue}Since $\pi_k(\fine_A(\surf))=0$ for all $k,$ we have that the constant map $g:\fine_A(\surf)\to *$ induces an isomorphism on all fundamental groups. Whitehead's theorem then implies that $g$ is a homotopy equivalence. We conclude that $\fine_A(\surf)$ is contractible. 
    % }
    %{\color{blue} Whitehead's theorem now implies that} $\fine_A(\surf)$ is weakly contractible. {\color{blue} Since $\fine_A(\surf)$ is a CW complex, we conclude that $\fine_A(\surf)$ is contractible.} %By Whitehead's theorem, since $\fine_A(\surf)$ is weakly contractible, it is contractible.
    %By Lemma~\ref{lemma:spheretostar}, $\phi(S^k)$ can be homotoped to be in the star of a vertex. Since $\fine_A(\surf)$ is a flag complex, the star of a vertex is contractible, meaning $\phi(S^k)$ is nullhomotopic.
%
    %By Whitehead's theorem, since $\fine_A(\surf)$ is a simplicial complex, we have that $\fine_A(\surf)$ is contractible, as desired.
        We will prove the theorem using Whitehead's theorem. Let $\phi\colon S^k\to |\fine_A(\surf)|$ be a continuous map. By Lemmas~\ref{lemma:nicesphere} and~\ref{lemma:spheretostar}, $\phi$ is homotopic to a simplicial map $\phi'\colon S^k\to |\fine_A(\surf)|$ whose image is in the star of a vertex. Since the star of a simplex in a simplicial complex is contractible,  $\phi'$ is nullhomotopic, and so $\fine_A(\surf)$ is weakly contractible. By Whitehead's theorem, since $\fine_A(\surf)$ is weakly contractible, it is contractible.
\end{proof}

\subsection{Proof of Theorem~\ref{thm:cchom}}\label{subsec:proofofcchom}

We now prove that the collapsing map $f:\fine(\surf)\to \mathcal{C}(\surf)$ is a homotopy equivalence.

\begin{proof}[Proof of Theorem~\ref{thm:cchom}]
   We note that the map $f:\fine(\surf)\to \cc(\surf)$ that sends curves to their isotopy classes is a simplicial map. Let $\sigma$ be a simplex in $\cc(\surf).$ Then $f^{-1}(\sigma)$ is contractible by Theorem~\ref{thm:singleisotopyclasscontractible}.
   Therefore, by Proposition~\ref{prop: contractible fib}
   %Applying work of Hatcher--Vogtmann \cite[Corollary 2.7]{HV}, it follows that
   $f$ is a homotopy equivalence.
\end{proof}

\section{Homotopy type of the fine arc complex}\label{sec:fac}

In this section, we prove that the fine arc complex is contractible (Theorem~\ref{thm:achom}). Although this theorem can be proved using the ideas in the proof of Theorem~\ref{thm:cchom}, we provide an alternate proof inspired by Hatcher flow in \textcite{Hatcher}. 

To do this, we use Whitehead's theorem. First we show that any map $S^{k}\to|\finearc(\surf)|$ is homotopic to a simplicial map whose image is contained in the star of a vertex.
%the image of any sphere of finite dimension under a continuous map can be homotoped so its image is contained in the star of a single vertex.
We follow Hatcher's proof idea in \cite{Hatcher} but make key modifications.

\begin{proof}[Proof of Theorem~\ref{thm:achom}]
    We will show that any map $\phi\colon S^{k}\to|\finearc(\surf)|$ is nullhomotopic by using Hatcher flow. Then we will apply Whitehead's theorem to conclude that $\finearc(\surf)$ is contractible.
    %We show that each sphere in $\finearc(\surf)$ is contractible using Hatcher flow and then apply Whitehead's theorem to conclude that $\finearc(\surf)$ is contractible.

    %Let each point in the arc complex be defined by its barycentric coordinates; that is, each point is a linear combination of a finite number of vertices $\sum_{i=0}^Nt_i\gamma_i$ with $t_i\geq 0$ for all $i$ and $\sum t_i =1.$ Consider a continuous map of a sphere $f:S^k\to\finearc(\surf)$. Since the image of $f$ is compact, standard properties of the CW topology imply that there is a finite collection of vertices, $\Gamma=\{\gamma_1,\ldots,\gamma_n\}$, such that each point in $S^k$ is mapped to some linear combination of $\gamma_1,\ldots,\gamma_n.$

    Since $S^{k}$ is compact, the image of $\phi$ is contained in the geometric realization of the full subcomplex of $\finearc(\surf)$ spanned by a finite set of arcs $\Gamma=\{\gamma_{1},\ldots, \gamma_{n}\}$.
     Let  $\beta\in \finearc(\surf)^{(0)}$ be an arc that intersects each $\gamma_i\in \Gamma$ at finitely many points such that 1) each intersection is crossing
     %2) $\beta \cap \mathcal{E}(\gamma_1,\ldots,\gamma_n) = \emptyset$,
     and 2) $\partial \beta\cap \gamma_i=\emptyset$ for all $i$. Such an arc is guaranteed to exist by Proposition \ref{lemma:keylemma}. 
    
     Recall that the star of $\beta$ consists of simplices $\{\alpha_{0},\ldots, \alpha_{p}\}$ in $\finearc(\surf)$ such that each $\alpha_{i}$ is disjoint from $\beta$ or equals $\beta$. 
     %that are collections of pairwise disjoint arcs that are also disjoint from $\beta.$ 
     The following exposition is directly inspired by that of Hatcher's proof in \cite{Hatcher}. Throughout, we point out the differences between our situation and that of Hatcher. 

    \p{Well-definedness in replacing arcs} During Hatcher flow, arcs are replaced with other arcs in such a way that intersection numbers with $\beta$ are reduced. We must take into account that, in our situation, we do not have the flexibility of isotoping arcs. We will therefore pre-determine how to replace arcs during Hatcher flow.

    We first orient $\beta.$ Since the number of intersections between $\beta$ and $\cup \gamma_i$ is finite, we may number the intersections in ascending order, beginning with the head of $\beta$ and proceeding toward the tail. We will now describe an algorithm for replacing arcs based on the intersections numbered above. 
    
    \pit{The first intersection} Say the first intersection is due to $\gamma_1,\ldots,\gamma_l$ (if $l\neq 1,$ we call this a \emph{multi-intersection}). Pick a neighborhood $N$ of $\beta$ between the intersection and the boundary component such that $N$ is disjoint from all other $\gamma_i$ and having the relevant portion of $\partial \surf$ and such that the arc(s) that meet at the first intersection point of $\beta$ are lie on the boundary of $N$, as in the left of Figure~\ref{fig:howtoarcs}. We then surger each $\gamma_1,\ldots,\gamma_l$ with $\partial N$, as on the right of Figure~\ref{fig:howtoarcs}. 

    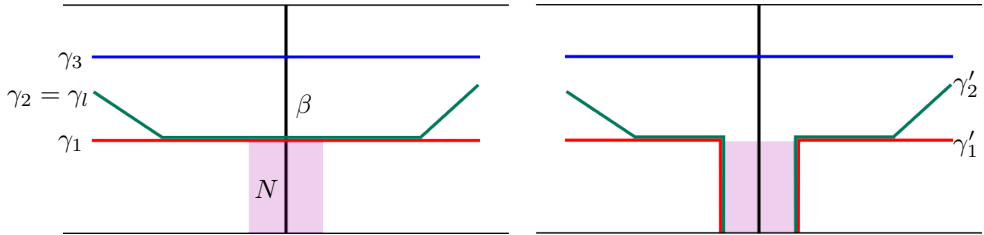
\begin{figure}[h]
    \begin{center}
    \begin{tikzpicture}
    \node[anchor = south west, inner sep = 0] at (0,0) {\includegraphics[width=0.8\textwidth]{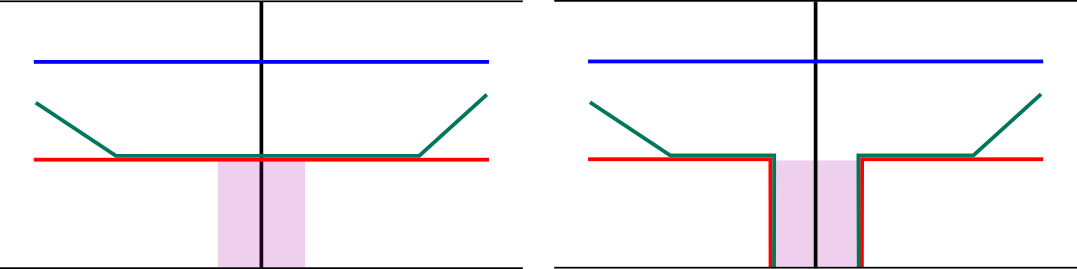}};
    %\draw[help lines] (0,0) grid (12,3);
    \node at (0.1,1.2) {$\gamma_1$};
    \node at (-0.2,1.8){$\gamma_2=\gamma_l$};
    \node at (0.1,2.3){$\gamma_3$};
    \node at (2.7,0.6){$N$};
    \node at (3.2,1.7){$\beta$};
    \node at (12,1.2){$\gamma'_1$};
    \node at (12,2){$\gamma'_2$};
    
    \end{tikzpicture}
    \caption{Left: $\beta$ appears vertically in black and there are several $\gamma_i$ crossing it. A neighborhood $N$ of the portion of $\beta$ between $\partial \surf$ and the first point of intersection is shaded in pink. Right: the first intersection is resolved via surgery with $\partial N.$}\label{fig:howtoarcs}
    \end{center}
    \end{figure}
    
    Thus, each arc $\gamma_1,\ldots,\gamma_l$ becomes two arcs after surgery. At least one has to be essential, since if they are both inessential, the case has to have been that the original $\gamma_i$ was itself inessential. We claim that for each $\gamma_i\in \gamma_1,\ldots, \gamma_l,$ there exists an essential choice for a surgered arc $\gamma'_i$ such that: 1) $\gamma_i\cap \gamma'_i=\emptyset,$ 2) if $\gamma_i\cap \gamma_j=\emptyset$, then $\gamma'_i\cap \gamma_j=\emptyset$ for $\gamma_j\in\Gamma$, 
    %3) $\gamma'_i\cap \beta \cap \gamma_j = \emptyset$ for all $\gamma_j\in \Gamma$, 
    3) $|\gamma'_i\cap \beta| \leq |\gamma_i\cap \beta|-1,$ and 4) $\gamma'_i$ is contained in a small neighborhood of $\partial N$ and $\gamma_i$.

    (1) and (2) can be guaranteed since there is an arc $\gamma'_i$ arbitrarily close to $\gamma_i$ that is also disjoint from $\gamma_i$ and any of the curves that $\gamma_i$ is disjoint from.  (3) uses that all intersections between $\gamma_i$ and $\beta$ are crossing.
    %For the remainder of the proof of the claim, we take $A$ to be a neighborhood of such a representative with $A$ disjoint from $\gamma_i$ and any $\gamma_j$ from which $\gamma_i$ is disjoint. (3) is accomplished inherently by the surgery and arises from the fact that all intersections between $\gamma_i$ and $\beta$ are crossing. 
    %(4) can be accomplished since all $\gamma_i$ are disjoint from $N$ after the surgery, and it is possible to isotope $\gamma_i$ in an arbitrarily small neighborhood of itself.
    (4) can be accomplished since it is possible to isotope $\gamma'_i$ in an arbitrarily small neighborhood of itself.

    We now replace the set $\Gamma$ with $\Gamma',$ in which each $\gamma_{i}$ in the set $\{\gamma_1,\ldots,\gamma_{l}\}$ is replaced with an arc $\gamma_{i}'$ by performing the above construction. We note that the total number of intersections with $\beta$ has been reduced. Moreover, the intersections between the curves in $\Gamma'$ and $\beta$ can once more be numbered. (We note that the numbering may be inconsistent with the prior enumeration based on which representative is chosen. However, all of the isotopies can be done in sufficiently small neighborhoods that the only differences in the numbering of the intersections is when a multi-intersection with $\beta$ splits up.)

    %Where relevant, we then replace $\gamma_i$ with $\gamma_i'$ .The remainder of the intersections with $\beta$ can again be numbered. We note that the numbering may be inconsistent with the prior enumeration based on which representative is chosen. However, all of the isotopies can be done in sufficiently small neighborhoods that the only differences in the numbering of the intersections is when a multi-intersection with $\beta$ splits up. We now replace the set $\Gamma$ with $\Gamma',$ in which each $\gamma_{i}$ in the set $\{\gamma_1,\ldots,\gamma_{l}\}$ is replaced with an arc $\gamma_{i}'$ by performing the above construction. We note that the total number of intersections with $\beta$ has been reduced.

    \pit{Past the first intersection} We now repeat the above with the new first intersection between $\Gamma'$ and $\beta.$ We continue to do this until no intersections remain between the surgered arcs and $\beta.$ This process terminates since the number of intersections strictly decreases with each iteration of the algorithm.
    
    The key point is that this algorithm is performed once a priori and the choices of surgered and isotoped arcs are fixed for the remainder of the proof. As a result, after sufficiently iterating this algorithm, we can construct a set of arcs $\Gamma''=\{\gamma_{1}'',\ldots, \gamma_{n}\}$ that satisfies: each $\gamma_{i}''$ is in the star of $\beta$; each $\gamma_i''$ was obtained via some surgery of $\gamma_i$ with $\beta;$ and if $\gamma_{i}\cap \gamma_{j}=\emptyset$, then $\gamma_{i}'' \cap \gamma_{j}'' =\emptyset$, for all $i, j=1,\ldots, n$. %As a result, after sufficiently iterating this algorithm, we can construct a set of arcs $\Gamma''=\{\gamma_{1}'',\ldots, \gamma_{n}\}$ that satisfies: 1) each $\gamma_{i}''$ is in the star of $\beta$, 2) $\gamma_{i} \cap \gamma_{i}'' =\emptyset$, and 3) if $\gamma_{i}\cap \gamma_{j}=\emptyset$, then $\gamma_{i}\cap \gamma_{j}'' =\emptyset$ and $\gamma_{i}'' \cap \gamma_{j}'' =\emptyset$, for all $i, j=1,\ldots, n$.

    \p{Performing Hatcher flow} We interpret a point $P$ in the image of $\phi$ as a weighted sum of the curves $\{\gamma_i\}_I$ by considering barycentric coordinates. We will define the flow of the image of $\phi$ by considering the image to lie in higher-dimensional simplices and flowing between faces.

    Consider the arcs arising from surgery with the arcs forming the first intersection, as in the Well-definedness in replacing arcs: The first intersection section. Let $P=\sum_{i=0}^m a_i\gamma_i \in \phi(S^k)$ be a point in a simplex spanned by $\gamma_1,\ldots,\gamma_m$ in terms of its barycentric coordinates. Suppose, without loss of generality, that $\gamma_1$ was one of the arcs forming the first intersection with $\beta.$ We then notice that the arcs $\gamma_1,\ldots,\gamma_m,\gamma'_1$ span a simplex as well, and we define a flow from the face containing $P$ to the face containing $\gamma'_1$ (but not $\gamma_1$) linearly: $H(P,t)=\sum_{i=2}^ma_i\gamma_i + (1-t)a_1\gamma_1 + ta'_1\gamma'_1$. We can extend this flow to all points $x\in\phi(S^k)$ by setting $x=\sum_{i=1}^na_i\gamma_i$ (using barycentric coordinates once more and potentially having many 0's as coefficients) and defining \[H(x,t)=\sum_{i=2}^na_i\gamma_i + (1-t)a_1\gamma_1 + ta'_1\gamma'_1.\] We note that any points not contained in the interior of simplices spanned (in part) by $\gamma_1$ are not moved.

    We then perform similar flows for the remainder of the arcs that are part of the first intersection with $\beta.$ (Further note that all of the flows corresponding to the first intersection could be performed simultaneously since no two of $\gamma_1,\ldots,\gamma_l$ can span a common simplex.)

    After performing the above flows, the image of $\phi$ is contained in simplices spanned by a new collection of arcs $\Gamma'=\gamma'_1,\ldots,\gamma'_l, \gamma_{l+1},\ldots,\gamma_n$, where the total intersection number of the arcs with $\beta$ has been reduced. We then repeatedly perform flows to reduce intersection number with $\beta$ until the image of $\phi$ is contained in the star of $\beta,$ as desired.

    % {\color{orange}ANOTHER ALTERNATIVE EXPOSITION FOR PERFORMAING HATCHER FLOW SECTION
    
    % \p{Performing Hatcher flow}
    % To construct Hatcher flow, we construct a homotopy $H\colon S^{k}\times I \to |\finearc(\surf)|$
    % by from $\phi$ to a map $\phi''$ whose image is in the full subcomplex spanned by vertices in $\Star(\beta)$ by using a homotopy similar to the one given in the proof of Lemma~\ref{lemma:nicesphere}. More specifically, without loss of generality we can assume that $\phi\colon S^{k}\to |\finearc(\surf)|$ is induced from a simplicial map $\phi\colon \mathbb{S}^{k}\to \finearc(\surf)$. For each vertex $v\in (\mathbb{S}^{k})^{(0)}$, let $\gamma_{v}= \phi(v)$. Let  $\phi''\colon \mathbb{S}^{k}\to \finearc(\surf)$ be the map sending a vertex $v\in (\mathbb{S}^{k})^{(0)}$ to $\gamma_{v}''$ (i.e. if $\gamma_{v}=\gamma_{1}$, then $\gamma_{v}'' = \gamma_{1}''$). Due to how $\Gamma''$ was constructed, $\phi''$ is a simplicial map. The map 
    % \begin{align*}
    %     H\colon |\mathbb{S}^{k}|\times I &\to |\finearc(\surf)| \\
    %     \big(\sum_{v\in (\mathbb{S}^{k})^{0}} t_{v} v, s\big)&\mapsto \sum_{v\in (\mathbb{S}^{k})^{0}} (1-s)t_{v} \gamma_{v}+ st_{v} \gamma_{v}''
    % \end{align*}
    % gives a homotopy from $\phi$ to $\phi''$. In particular, $\phi$ is homotopic to a map whose image is contained in $\Star(\beta)$. Therefore, $\finearc(\surf)$ is weakly contractible, and so by Whitehead's theorem, $\finearc(\surf)$ is contractible.}

    \medskip \noindent Now that we have that the image of every sphere under a continuous map is homotopic to one contained in the star of a single vertex and is therefore nullhomotopic, we invoke Whitehead's theorem to conclude that $\finearc(\surf)$ is contractible for the prescribed surfaces.
\end{proof}

 \pit{Distinctions from Hatcher's proof} Our difference in setting requires a few changes to Hatcher's argument in \cite{Hatcher}, which we briefly mention here for the convenience of experts.
 
 One distinction between our proof and that of Hatcher is that we have an uncountable number of choices whenever a surgery is performed. This is taken care of by the ``Well-definedness in replacing arcs'' section of the proof. 
    
Another distinction is that Hatcher asks for there to be no bigons between any pair of arcs. Hatcher needs this since his proof is up to isotopy, but we remove this requirement by ensuring that all intersections are crossing and that all arcs are still surgered to create at least one essential arc. 

     We also explicitly mention here that although there were simultaneous surgeries in the ``Well-definedness when replacing arcs'' section, such a phenomenon does not actually occur during the flow. %we may be surgering with two arcs simultaneously (if the arcs coincide in a neighborhood of an intersection with $\beta$), but this does not affect the proof.

    One of the clearest distinctions between our proof and that of Hatcher is that Hatcher prescribes a specific set of endpoints $V\subset \partial \surf$ for his arcs and asks for no two elements of $V$ to be in the same connected component of the boundary. Hatcher treats the points in $V$ as punctures that cannot be isotoped across, so an essential arc could bound a disk with a marked boundary in $\surf.$ If this is the case, then surgery via the flow argument could result in no essential arcs. However, such arcs are not permitted in our case, as we do not treat $\partial \surf$ as having marked points. As such, we impose no further restrictions on the arcs that comprise the vertices of the fine arc graph.

\printbibliography

\end{document}